\documentclass[a4paper,11pt,reqno]{amsart}
\usepackage[utf8]{inputenc}
\pdfoutput=1

\usepackage{hyperref}

\usepackage[sc]{mathpazo}
\usepackage[tracking]{microtype}

\usepackage[T1]{fontenc}
\usepackage{textcomp}
\usepackage{amssymb}

\usepackage[arrow,matrix]{xy}

\newtheorem{thm}{Theorem}
\newtheorem{prop}{Proposition}[section]
\newtheorem{lem}[prop]{Lemma}

\theoremstyle{definition}

\theoremstyle{remark}

\newtheorem{claim}[prop]{Claim}

\DeclareMathOperator{\ch}{ch}
\DeclareMathOperator{\Hom}{Hom}
\DeclareMathOperator{\Pic}{Pic}

\newcommand{\bF}{{\mathbb{F}}}
\newcommand{\bP}{{\mathbb{P}}}
\newcommand{\bZ}{{\mathbb{Z}}}
\newcommand{\sO}{{\mathcal{O}}}

\begin{document}

\title[Contractions of 2-Fanos]{Extremal contractions of 2-Fano fourfolds}

\author{Florian Schrack}
\address{Mathematisches Institut\\ Universität Bayreuth\\ 95440 Bayreuth\\ Germany}
\email{florian.schrack@uni-bayreuth.de}


\date{June 12, 2014}

\begin{abstract}
We consider extremal contractions on smooth Fano fourfolds whose second Chern character is positive. We show that such contractions can neither be of fiber type nor contract a divisor to a point.
\end{abstract}

\maketitle
\section{Introduction}
A Fano manifold~$X$ is called \emph{2-Fano}, if its second Chern character $\ch_2(X) := \ch_2(T_X)$ is positive, i.e., $\ch_2(X).S > 0$ for all $S\in \overline{NE}_2(X)\setminus\{0\}$. We recall that for a vector bundle~$E$, its second Chern character~$\ch_2(E)$ is given by
\[ \ch_2(E) = \tfrac{1}{2}\bigl(c_1^2(E)-2c_2(E)\bigr). \]

2-Fano manifolds have first been introduced by de Jong and Starr in \cite{dJS06} and \cite{dJS07} in order to obtain a natural sufficient condition for the technical notion of \emph{rational simple connectedness}. The property of rational simple connectedness is used in~\cite{dJHS11} to prove a generalization to the surface case of the Graber--Harris--Starr theorem on the existence of sections in families of rationally connected varieties over curves \cite{GHS03}.

The classification of 2-Fano manifolds is still very much an open problem. In dimension~$3$, as shown by Araujo--Castravet in~\cite{AC12}, the only 2-Fano manifolds are $\bP_3$ and $Q_3$. For $n:=\dim X \ge 4$, \cite{AC12} gives a classification of 2-Fano manifolds with index $\ge n-2$, but for arbitrary index, the overall picture remains unclear.

All known examples of $2$-Fanos seem to have second Betti number one, so a natural question to ask is whether it is possible for a $2$-Fano manifold~$X$ to have $b_2(X) \ge 2$.

The present paper studies this question in dimension~$4$, i.e., we let $X$ be a $2$-Fano fourfold and assume that $b_2(X)\ge 2$. Then, by Mori theory, there exists an extremal contraction
\[
f\colon X \to Y,
\]
where $Y$ is a normal projective variety with $1\le \dim Y \le 4$.

We first study the cases where $\dim Y < 4$. In section~\ref{sec:conic} we  exclude the case $\dim Y = 3$ by showing that a 2-Fano manifold $X$ of arbitrary dimension cannot have an extremal contraction to a variety of dimension $\dim X - 1$ (cf.~Proposition~\ref{prop:conicbund}). Sections \ref{sec:p2fib} and \ref{sec:p3q3fib} deal with the cases $\dim Y = 2$ and $\dim Y = 1$, respectively. We thus arrive at the following theorem:
\begin{thm}
Let $X$ be a smooth 2-Fano fourfold and $f\colon X \to Y$ the contraction of an extremal ray in~$\overline{NE}(X)$. Then $f$ is birational.
\end{thm}
Now in the birational case, the situation is more complicated, so we cannot give a complete result. Nevertheless, we do some calculations for blow-ups in section~\ref{sec:bir} and can thus rule out one important case of birational contractions:
\begin{thm}
Let $X$ be a smooth 2-Fano fourfold and $f\colon X \to Y$ a birational contraction of an extremal ray in~$\overline{NE}(X)$. Then $f$ cannot contract a divisor to a point.
\end{thm}
\subsection*{Acknowledgments}
I would like to thank Thomas Peternell for bringing the subject at hand to my attention and for many useful discussions and suggestions.

\section{Conic bundles}\label{sec:conic}
Let $Z$, $U$ be smooth varieties and $g\colon Z \to U$ a conic bundle (i.e., a proper morphism such that each fiber of~$g$ is isomorphic to a conic in~$\bP_2$). Then there exists an effective divisor $\Delta_g \subset U$ such that $Z_u := g^{-1}(u)$ is smooth if and only if $u\not\in\Delta_g$.
\begin{lem}\label{lem:conicbund}
In the situation given above, for any smooth proper curve $C\subset U$ such that $S:=g^{-1}(C)$ is smooth, we have
\[ \ch_2(T_Z|_S) = -\tfrac{3}{2} \deg \Delta_g|_C \le 0.\]
\end{lem}
\begin{proof}
Since $g$ is flat, we have $N^*_{S|Z}=(g|_S)^*N^*_{C|U}$ and thus $\ch_2(N^*_{S|Z}) = (g|_S)^*\ch_2(N^*_{C|U})$, which is zero because $\dim C=1$.
So from the exact sequence
\[
0 \longrightarrow T_S \longrightarrow T_Z|_S \longrightarrow N_{S|Z} \longrightarrow 0,
\]
we obtain
\[
\ch_2(T_Z|_S) = \ch_2(S).
\]
This shows that we can assume from now on that $U=C$ and $Z=S$.

We now consider the rank-3 vector bundle $E:= g_*\sO_Z(-K_Z)$ and the projective bundle
\[
\pi\colon \bP(E) \to U.
\]
An easy calculation shows that $Z$ can be embedded into~$\bP(E)$ as a divisor
\[
Z \in \lvert \sO_{\bP(E)}(2) + \pi^*(\det E^* -K_U ) \rvert.
\]
Equivalently, $Z$ is given by a section
\[
s \in H^0(S^2E\otimes\det E^*\otimes \sO_U(-K_U)) \subset \Hom(E^*, E\otimes\det E^*\otimes \sO_U(-K_U)).
\]
From this, it follows that
\[
\Delta_g \in \lvert \det E^* -3K_U\rvert.
\]
On the other hand, by \cite[Lem.~4.1]{dJS06}, we have
\[
\ch_2(\bP(E)) = {\tfrac{3}{2}}\xi^2 + \pi^*c_1(E^*).\xi,
\]
where $\xi := c_1(\sO_{\bP(E)}(1))$. It follows that
\[
\begin{split}
\ch_2(Z) &= \bigl(\ch_2(\bP(E)) - \tfrac{1}{2}(2\xi + \pi^*c_1(E^*)+\pi^*c_1(U))^2\bigr)|_Z\\
&= \bigl(-\tfrac{1}{2}\xi^2-\pi^*c_1(E^*).\xi-2\pi^*c_1(U).\xi\bigr)\bigl(2\xi+\pi^*c_1(E^*)+\pi^*c_1(U)\bigr)\\
&= -\tfrac{3}{2}c_1(E^*)-\tfrac{9}{2}c_1(U) = -\tfrac{3}{2} \deg \Delta_g.
\end{split}
\]
Since $Z$ and $U$ are smooth, the general fiber of~$g$ is smooth, so $\Delta_g$ is effective, i.e., $\deg\Delta_g\le 0$.
\end{proof}
\begin{prop}\label{prop:conicbund}
Let $X$ be a smooth projective variety with $n:=\dim X \ge 2$ and $f\colon X \to Y$ the contraction of an extremal ray on~$X$. Suppose $\dim Y = n-1$. Then there exists a smooth surface $S\subset X$ such that $\ch_2(X).S \le 0$.
\end{prop}
\begin{proof}
There exists a subvariety $A\subset Y$ of codimension $\ge 2$ such that if we let $U:=Y\setminus A$, then $U$ is smooth and
\[
g := f|_{f^{-1}(U)}\colon f^{-1}(U) \to U
\]
is equidimensional. By~\cite[Thm.~3.1]{And85} (cf.~\cite[Prop.~2.2]{Kac97}), this implies that $g$ is a conic bundle. Now if we take $n-2$ general ample divisors $H_1$, $\dots$, $H_{n-2}$ on $Y$, then $H_1\cap \dots \cap H_{n-2}$ is contained in~$U$ and both $H_1\cap \dots \cap H_{n-2}$ and
\[
S:=g^{-1}(H_1\cap \dots \cap H_{n-2})
\]
are smooth. Then we have $\ch_2(X).S=\ch_2(T_{f^{-1}(U)}|_S)\le 0$ by Lemma~\ref{lem:conicbund}.
\end{proof}

\section{$\bP_2$-fibrations}\label{sec:p2fib}
We now consider extremal contractions
\[
f\colon X \to Y
\]
where $X$ is a 2-Fano fourfold and $Y$ is a surface. We will show in this section that the above situation cannot occur.

We start by observing that a general fiber~$F$ of~$f$ is a smooth surface with $N_{F|X}\cong \sO_{F}^{\oplus2}$, so $-K_F=-K_X|_F$ is ample and $\ch_2(F) = \ch_2(X).S > 0$. In other words, $F$ is 2-Fano, and thus $F \cong \bP_2$.

By~\cite[Cor.~(1.4)]{AW97}, $Y$ is smooth. So, since $X$ is Fano, $Y$ must be a blow-up of either~$\bP_2$ or a Hirzebruch surface. In particular, there exists a base point free linear system on~$Y$ whose general element~$\ell$ is isomorphic to~$\bP_1$.

If we consider the preimage $X_\ell := f^{-1}(\ell)$, then $X_\ell$ is a smooth threefold with the property that $\ch_2(X_\ell).S > 0$ for any $S\in \overline{NE}_2(X_\ell)\setminus\{0\}$.

We now apply the following lemma:
\begin{lem}\label{lem:p2fib}
Let $g\colon Z \to T$ be a surjective morphism from a smooth projective threefold~$Z$ to a smooth curve~$T$.
Assume that
\begin{enumerate}
\item $-K_Z$ is $g$-ample,
\item $\ch_2(Z).S > 0$ for all irreducible surfaces $S\subset Z$ and
\item a general fiber of~$g$ is isomorphic to~$\bP_2$.
\end{enumerate}
Then any singular fiber~$F_{\mathrm{sing}}$ of~$g$ can (as a divisor) be decomposed as
\[
F_{\mathrm{sing}} = F_1 + F_2 + F_3,
\]
where $F_i\cong \bF_1$ for $i=1$, $2$, $3$ and $F_i\cap F_j$ is a line for $i\ne j$ (either a fiber of the ruling or the exceptional section of~$\bF_1$).

More precisely, there is a finite sequence of birational morphisms
\[ Z = Z_0 \stackrel{\alpha_1}{\longrightarrow} Z_1 \stackrel{\alpha_2}{\longrightarrow}
\dots \stackrel{\alpha_s}{\longrightarrow} Z_s,\]
where $Z_s$ is a $\bP_2$-bundle over~$T$ and each~$\alpha_i$ is given as the composition
\[ \alpha_i = \gamma_i\circ\beta_i,\]
where $\beta_i$ is the blow-up of a line in a smooth fiber of~$Z_i$ over $T$ and $\gamma_i$ is the blow-up of a non-trivial fiber of~$\beta_i$.

\end{lem}
\begin{proof}
The result is a corollary of the more general classification of Del~Pezzo fibrations over curves carried out by T.~Fujita in~\cite{Fuj90}. In our special case, the proof is simplified considerably by ruling out most cases using the positivity assumption on the second Chern character as we will do in the following.

We let $F_{\mathrm{sing}}$ be a singular fiber of~$g$. Then by~\cite[Thm.~5]{ARM12}, $F_{\mathrm{sing}}$ is reducible. So, as a divisor, we can decompose $F_{\mathrm{sing}}$ as
\[
F_{\mathrm{sing}} = \sum_{j=1}^r m_jF_j,
\]
where $r\ge2$, the $F_j$ are the irreducible components of~$F_{\mathrm{sing}}$ and $m_j\ge1$. Then we have the following
\begin{lem}\label{lem:compextr}
For any irreducible component $F_j$ of $F_{\mathrm{sing}}$, there exists an irreducible curve $C_j\subset F_j$ with $C_j.F_j < 0$ whose numerical class~$[C_j]$ is extremal in~$\overline{NE}(Z/T)$.
\end{lem}
\begin{proof}
Since the fiber $F_{\mathrm{sing}}$ is connected, the irreducible component~$F_j$ meets another irreducible component~$F_k$ with $k\ne j$. Let $\tilde{C}\subset F_j$ be a curve meeting $F_k$ but not contained in any $F_i$ for $i\ne j$. Then clearly
\[
F_k.\tilde{C} > 0\qquad\text{and}\qquad F_i . \tilde{C} \ge 0\quad\text{for $i\ne j$}.
\]
Since $F_{\mathrm{sing}}.\tilde{C} = 0$, this implies
\[
F_j.\tilde{C} < 0.
\]
Now since $-K_Z$ is $g$-ample, by the relative cone theorem, $\overline{NE}(Z/T)$ is generated by the classes of finitely many irreducible extremal curves, so in particular, we can write
\[
\tilde{C} \equiv \sum_{l=1}^m \alpha_l \tilde{C}_l,
\]
where the $\tilde{C}_l$ are irreducible extremal curves and $\alpha_l >0$.
From $F_j.\tilde{C} < 0$, it then follows that there exists an $l\in\{1,\dots, m\}$ such that $F_j.\tilde{C}_l < 0$ and thus $C_j:=\tilde{C}_l \subset F_j$.
\end{proof}

By the relative contraction theorem, there exists for each~$j$ a relative Mori contraction of~$g$ over~$T$ contracting the $C_j\subset F_j$ constructed in the preceding lemma, i.e., a Mori contraction $\phi_j\colon Z \to Z'_j$, where $Z'_j$ is a normal projective variety, and a morphism $g'_j\colon Z'_j \to T$ such that $g_j'\circ \phi_j = g$.

If $\phi_j$ contracts any curve contained in a general fiber of~$g$, then $\phi_j = g$, thus $g$ is itself a Mori contraction and thus a $\bP_2$-bundle by~\cite[Thm.~3.5]{Mor82}. This is a contradiction to the reducibility of~$F_{\mathrm{sing}}$, so $\phi_j$ must be birational.

Thus, by~\cite[Thm.~3.3]{Mor82}, there is an irreducible divisor~$D_j$ on~$Z$ contained in a fiber of~$g$ such that $\phi_j$ contracts~$D_j$ and $\phi_j$ is an isomorphism on~$Z\setminus D_j$. For any $i\ne j$, any curve $C'\subset F_i$ with $C'\not\subset F_j$ satisfies $C'.F_j\ge 0$, so $C'$ is not numerically proportional to~$C_j$. In particular, $\phi_j$ cannot contract $F_i$. Thus we have shown that $D_j = F_j$, so again by~\cite[Thm.~3.3]{Mor82}, we obtain the following types of possible fiber components~$F_j$:
\begin{enumerate}
\item[(T1)] $F_j \cong \bP_2$ with normal bundle $N_{F_j|Z}\cong \sO_{\bP_2}(-k)$, $k\in\{1,2\}$,
\item[(T2)] $F_j \cong \bP_1\times\bP_1$ with $N_{F_j|Z} \cong \sO_{\bP_1\times\bP_1}(-1,-1)$,
\item[(T3)] $F_j \cong Q$, where $Q\subset\bP_3$ is a quadratic cone and $N_{F_j|Z}\cong \sO_{\bP_3}(-1)|_Q$,
\item[(T4)] $Z_j'$ is smooth and there exists a smooth curve $C_j'\subset Z_j'$ such that $\phi_j$ is the blow-up of~$C_j'$. Thus, $F_j \cong \bP(N_{C_j'|Z_j'}^*)$ and $N_{F_j|Z}\cong \sO_{\bP(N_{C_j'|Z_j'}^*)}(-1)$.
\end{enumerate}
We further note that in types (T1), (T2) and (T3), the fiber component~$F_j$ is mapped to a point by the contraction morphism~$\phi_j$.

By straightforward computations, we can obtain the intersection number $\ch_2(Z).F_j$ in each of these cases:
\begin{enumerate}
\item[(T1)] $\ch_2(Z).F_j = \tfrac{3}{2} + \tfrac{1}{2}k^2$,
\item[(T2)] $\ch_2(Z).F_j = 1$,
\item[(T3)] $\ch_2(Z).F_j = 1$,
\item[(T4)] $\ch_2(Z).F_j = -\tfrac{1}{2} \deg N_{C_j'|Z_j'}$.
\end{enumerate}
By assumption, we have $\ch_2(Z).F_j>0$ for all~$j$. Furthermore, since
$\ch_2(Z).F_{\mathrm{sing}} = \ch_2(\bP_2) = \tfrac{3}{2}$, we have
\[
\sum_{j=1}^r m_i \ch_2(Z).F_j = \tfrac{3}{2}.
\]
Combining this with the above calculations, we see that only the following possibilities remain (up to re-numbering the~$F_j$):
\begin{enumerate}
\item[(P1)] $r=2$, $m_1 = 1$, $F_1$ is of type~(T4) with $\deg N_{C_1'|Z_1'} = -1$, and $m_2 \in \{1, 2\}$ with $F_2$ of type (T2), (T3) or (T4).
\item[(P2)] $r=3$, $m_1 = m_2 = m_3 = 1$, and each $F_j$ is of type~(T4) with $\deg N_{C_j'|Z_j'} = -1$.
\end{enumerate}
We now show that possibility~(P1) cannot occur. We assume that a fiber~$F_{\mathrm{sing}}$ of~$g$ satisfies~(P1). We consider the contraction~$\phi_1\colon Z\to Z_1'$ constructed above and let $F_{\mathrm{sing}}'=\phi_1(F_{\mathrm{sing}})$. We know that $Z_1'$ is smooth, which implies that $g_1'$ is flat. Since $F_{\mathrm{sing}}'$ is irreducible, we can apply~\cite[Thm.~5]{ARM12} to conclude that $F_{\mathrm{sing}}'\cong\bP_2$. But now $\phi_1$ is the blow-up of a smooth curve $C_1'\subset F_{\mathrm{sing}}'$, so in particular, $F_2$, which is the strict transform of $F_{\mathrm{sing}}'$ in~$Z$, is isomorphic to~$F_{\mathrm{sing}}'\cong\bP_2$, contradicting (T2), (T3) and (T4).

It remains to study possibility~(P2) in greater detail. We let $F_{\mathrm{sing}}=F_1+F_2+F_3$ be a fiber of~$g$ that satisfies~(P2). We first note that since $-K_Z$ is $g$-ample, we have
\[ F_i.F_j.(-K_Z) \ge 0\quad\text{for $i\ne j$} \]
and since $F$ is connected, we can for each~$j$ even find an $i_0\ne j$ such that
\[ F_{i_0}.F_j.(-K_Z) > 0. \]
Together with $F_j.(F_1+F_2+F_3) = 0$, this implies that
\begin{equation}\label{eq:negselfint}
F_j^2.(-K_Z) < 0 \quad\text{for $j=1$, $2$, $3$.}
\end{equation}

By applying Kodaira vanishing and using~(T4) and the adjunction formula, we obtain
\begin{equation}\label{eq:vaneuler}
\chi(N_{F_j|Z}) = \chi(K_{F_j} - K_Z|_{F_j}) = 0.
\end{equation}
Riemann-Roch yields
\[ \chi(N_{F_j|Z}) = \tfrac{1}{2}(F_j|_{F_j})^2-\tfrac{1}{2}K_{F_j}.F_j|_{F_j} + \chi(\sO_{F_j}) = \tfrac{1}{2}F_j^2.(-K_Z) + \chi(\sO_{F_j}).\]
Applying \eqref{eq:negselfint} and \eqref{eq:vaneuler}, we obtain
\[ \chi(\sO_{F_j}) > 0. \]
Since obviously $h^2(\sO_{F_j}) = h^0(K_{F_j}) = 0$, we conclude that $h^1(\sO_{F_j}) = 0$ and thus the curve $C_j'\subset Z_j'$ from~(T4) must be isomorphic to~$\bP_1$.

We now let $\tilde{F}_j=\phi_j(F_{\mathrm{sing}})$. Then we have the following normal bundle sequence:
\begin{equation}\label{eq:normbund}
0 \longrightarrow N_{C_j'|\tilde{F}_j} \longrightarrow N_{C_j'|Z_j'} \longrightarrow N_{\tilde{F}_j|Z_j'}|_{C_j'} \longrightarrow 0.
\end{equation}
But now $N_{\tilde{F}_j|Z_j'}$ is trivial, so from $\deg N_{C_j'|Z_j'}=-1$, we infer that $N_{C_j'|\tilde{F}_j}\cong\sO_{\bP_1}(-1)$ and hence
\[ N_{C_j'|Z_j'} \cong \sO_{\bP_1}(-1) \oplus \sO_{\bP_1}. \]
So in particular, $F_j \cong \bF_1$.

We consider again the contraction $\phi_1\colon Z \to Z_1'$, which is the blow-up of a smooth curve $C_1'\subset Z_1'$. If we let $F_2'$ and $F_3'$ be the images of $F_2$ and $F_3$ via $\phi_1$, respectively, we can assume that $C_1'\subset F_2'$. Then $F_2$ is mapped isomorphically onto $F_2'$ by~$\phi_1$.

We now want to apply Lemma~\ref{lem:compextr} to~$Z_1'$ in order to find a contraction morphism $\psi_2\colon Z_1' \to Z_2''$ which contracts~$F_2'$. In order to do this, we have to show that $-K_{Z_1'}$ is $g_1'$-ample. Since $\overline{NE}(Z/T)$ is generated by finitely many curves, the same is true for~$\overline{NE}(Z_1'/T)$. So it suffices to show that $-K_{Z_1'}.C > 0$ for any curve $C$ contained in a fiber of~$g_1'$. But now since $\phi_1$ is the blow-up of~$C_1'$, we have
\[ K_Z = \phi_1^*K_{Z_1'} + F_1, \]
so $-K_{Z_1'}.C > 0$ is certainly true for all irreducible $C \ne C_1'$ since such $C$ can be obtained as the image of a curve $\tilde{C}\subset Z$ with $\tilde{C} \not\subset F_1$. But now also $-K_{Z_1'}.C_1' > 0$, which one can see as follows: Let $C_e \subset Z$ be the exceptional section of $F_1 \to C_1'$. Then $\phi_{1*}C_e = C_1'$ and $K_{F_1}|_{C_e} \cong \sO_{\bP_1}(-1)$, hence by adjunction
\[ F_1.C_e = -K_Z.C_e - 1 \ge 0 \]
and thus $-K_{Z_1'}.C_1' > 0$ as in the previous case.

So by Lemma~\ref{lem:compextr}, we obtain a relative contraction morphism $\psi_2\colon Z_1' \to Z_2''$ contracting $F_2'$ as desired. We let $g_2''\colon Z_2'' \to T$ be the induced morphism. Since the exceptional divisor of $\psi_2$ is $F_2' \cong \bF_1$ we conclude by~\cite[Theorem~3.3]{Mor82} that $Z_2''$ is smooth and $\psi_2$ is the blow-up of a smooth curve~$C_2''\subset Z_2''$. By the same argument as in the discussion of possibility~(P1) above, we see that the fiber of $g_2''$ over $g(F_{\mathrm{sing}})\in T$ must be isomorphic to~$\bP_2$. If we consider the normal bundle sequence for $C_2''\subset \psi_2(\phi_1(F_{\mathrm{sing}})) \subset Z_2''$ (cf.~sequence~\eqref{eq:normbund}), we can conclude from $\bF_1 \cong F_2' \cong \bP(N_{C_2''|Z_2''})$ that $C_2''$ must be a line in~$\bP_2$.

We conclude by showing that $C_1'$ must be a fiber of $F_2'\to C_2''$. We first note that by adjunction, we have
\[
K_{Z_1'}.C_1' = \deg K_{C_1'} - \deg N_{C_1'|Z_1'} = -1.
\]
Furthermore, obviously $F_3'=\phi_1(F_3)$ is the strict transform of $\psi_2(F_3')\cong \bP_2$, so $F_3'\cong\bP_2$. Since $F_3\cong\bF_1$ is the blow-up of $F_3'$ in $C_1'\cap F_3'$, we conclude
that $F_3'.C_1'=1$ and thus $F_2'.C_1'=-1$ since $(F_2'+F_3').C_1'=0$.
Finally, since $K_{Z_1'} = \psi_2^*K_{Z_2''}+F_2'$, it follows that $\psi_2^*K_{Z_2''}.C_1' = 0$, from which we conclude that $\psi_2(C_1')$ must be a point since $-K_{Z_2''}$ is obviously $g_2''$-ample.
\end{proof}

We now apply Lemma~\ref{lem:p2fib} to~$X_\ell$. We get a finite sequence
\[ X_\ell = Z_0 \stackrel{\alpha_1}{\longrightarrow} Z_1 \stackrel{\alpha_2}{\longrightarrow}
\dots \stackrel{\alpha_s}{\longrightarrow} Z_s\]
of birational morphisms as described in the lemma, where $Z_s \to \ell\cong\bP_1$ is a $\bP_2$-bundle. We now apply to~$Z_s$ the following lemma:

\begin{lem}\label{lem:projbundp1}
Let $E$ be a vector bundle on~$\bP_1$ of rank~$r\ge 2$. Then there exists a rank-$2$ quotient bundle $F$ of $E$ such that the surface
\[ S := \bP(F) \subset \bP(E) \]
satisfies $\ch_2(\bP(E)).S \le 0$.
\end{lem}
\begin{proof}
By Grothendieck, there exist integers $a_1 \le \dotsb \le a_r$ such that
\[ E \cong \sO_{\bP_1}(a_1) \oplus \dotsb \oplus \sO_{\bP_1}(a_r). \]
We consider the quotient
\[ E \twoheadrightarrow \sO_{\bP_1}(a_1) \oplus \sO_{\bP_1}(a_2) =: F \]
given by the projection onto the first two summands of~$E$. Then we have a natural inclusion $S:=\bP(F) \subset \bP(E)$ and if we let $\zeta := \sO_{\bP(E)}(1)|_{\bP(F)}$, the normal bundle is given by
\[ N_{S|\bP(E)} \cong \zeta \otimes \pi^*{E^*/F^*} \cong \zeta \otimes \pi^*\bigl(\sO_{\bP_1}(-a_3) \oplus \dotsb \oplus \sO_{\bP_1}(-a_r)\bigr),\]
where $\pi\colon S=\bP(F) \to \bP_1$ denotes the natural projection. We now consider the sequence
\[ 0 \longrightarrow T_S \longrightarrow T_{\bP(E)}|_S \longrightarrow N_{S|\bP(E)} \longrightarrow 0. \]
Since $F$ has rank~$2$, we have $\ch_2(T_S) = 0$ by~\cite[Prop.~4.3]{dJS06}, so we obtain from the above sequence:
\[ \begin{split}
\ch_2(\bP(E)).S &= \ch_2(N_{S|\bP(E)})\\
&= \tfrac{1}{2}\zeta^2\cdot(r-2) + \zeta.\pi^*\sO(-a_3-\dotsb-a_r)\\
&= \tfrac{1}{2}(r-2)(a_1+a_2)-a_3-\dotsb-a_r \le 0 \qedhere
\end{split} \]
\end{proof}

We let $S_s := S\subset Z_s$ be the surface constructed in Lemma~\ref{lem:projbundp1}. We denote by $S_k$, $k=0$, $\dotsc$, $s-1$, the strict transform of~$S$ inside~$Z_k$.
\begin{claim}\label{cl:line}
For any $k=0$, $\dotsc$, $s$, the intersection of $S_k$ with any smooth fiber of~$Z_k$ over~$\ell$ is a line in~$\bP_2$. Furthermore, $\ch_2(Z_k).S_k \le 0$.
\end{claim}
To prove Claim~\ref{cl:line}, we argue by induction over~$k$. The case $k=s$ follows from Lemma~\ref{lem:projbundp1}. Now we assume by induction that Claim~\ref{cl:line} is true for some~$k\ge1$. We consider the birational morphism $\alpha_k\colon Z_{k-1}\to Z_k$, which by Lemma~\ref{lem:p2fib} can be decomposed as
\[ \alpha_k = \gamma_k\circ\beta_k, \]
where $\beta_k\colon \tilde{Z}_k \to Z_k$ is the blow-up of a line $L_1$ contained in a smooth fiber of~$Z_k$ over~$\ell$ and $\gamma_k\colon Z_{k-1}\to\tilde{Z}_k$ is the blow-up of a one-dimensional fiber $L_2\cong\bP_1$ of~$\beta_k$. We let $\tilde{S}_k$ be the strict transform of $S_k$ in~$\tilde{Z}_k$ and consider different cases.

The first case is that $L_1 \not\subset S_k$. Then $L_1 \cap S_k$ is a simple point by the induction hypothesis, so in particular $\tilde{S}_k$ is the blow-up of~$S_k$ in the point~$L_1 \cap S_k$. Furthermore, by~\cite[Lem.~4.13(i)]{AC12},
\[ \ch_2(\tilde{Z}_k).\tilde{S}_k = \ch_2(Z_k).S_k-\tfrac{3}{2} < 0. \]
Now if we blow up $L_2$ in this case, we either have $L_2 \not\subset\tilde{S}_k$, whence
\[ \ch_2(Z_{k-1}).S_{k-1} \le \ch_2(\tilde{Z}_k).\tilde{S}_k < 0 \]
by~\cite[Lem.~4.13(i)]{AC12}, or we have $L_2 \subset \tilde{S}_k$. But then $L_2$ is a $(-1)$-curve in~$\tilde{S}_k$ by the above considerations, so
\[ \ch_2(Z_{k-1}).S_{k-1} = \ch_2(\tilde{Z}_k).\tilde{S}_k - \tfrac{3}{2} + 1 < 0 \]
by~\cite[Lem.~4.13(ii)]{AC12}.

It remains to consider the case $L_1\subset S_k$. Then $L_1$ is a fiber of~$S_k$ over~$\ell$, so by~\cite[Lem.~4.13(ii)]{AC12},
\[
\ch_2(\tilde{Z}_k).\tilde{S}_k = \ch_2(Z_k).S_k - 1 < 0.
\]
But now since $L_1\subset S_k$, the blow-up morphism~$\beta_k$ maps $\tilde{S}_k$ isomorphically onto~$S_k$, so in particular $L_2 \not\subset \tilde{S}_k$. This implies by~\cite[Lem.~4.13(i)]{AC12} that
\[ \ch_2(Z_{k-1}).S_{k-1} \le \ch_2(\tilde{Z}_k).\tilde{S}_k < 0. \]
Thus we have inductively constructed an irreducible (smooth) surface $S \subset X_\ell = Z_0$ with $\ch_2(X_\ell).S \le 0$ (even $\ch_2(X_\ell).S < 0$ if $X_\ell$ is not a $\bP_2$-bundle). Now since $N_{X_\ell|X}=(f|_{X_\ell})^*N_{\ell|Y}$, we have $\ch_2(N_{X_\ell|X})=0$, so it follows that
\[ \ch_2(X).S \le 0, \]
therefore $X$ is not $2$-Fano.

\section{$\bP_3$- and $Q_3$-fibrations}\label{sec:p3q3fib}
In this section we deal with the case of a smooth projective fourfold~$X$ admitting an extremal contraction
\[f\colon X\to C\]
over a curve~$C$. Then if $X$ is 2-Fano, a general fiber~$F$ of~$f$ must also be 2-Fano (cf.~the beginning of section~\ref{sec:p2fib}), so either $F\cong\bP_3$ or $F\cong Q_3$ (where $Q_3 \subset \bP_4$ is a smooth quadric hypersurface) by~\cite[Thm.~1.2]{AC12}. In the first case we easily obtain:
\begin{prop}
In the above situation, assume that the general fiber of~$f$ is $\bP_3$. Then $X$ is not 2-Fano.
\end{prop}
\begin{proof}
Since the relative Picard number of~$f$ is $1$, every fiber of~$f$ is irreducible. By~\cite[Thm.~5]{ARM12}, this implies that $f$ is a $\bP^3$-bundle. Since $C$ is a curve, there exists a rank-4 vector bundle on~$C$ such that $X\cong \bP(E)$. The proposition now follows from~\cite[Cor.~4.6]{dJS06}.
\end{proof}

It remains to consider the case where $F\cong Q_3$:
\begin{prop}
In the above situation, assume that the general fiber of~$f$ is $Q_3$. Then $X$ is not $2$-Fano.
\end{prop}
\begin{proof}
Since the relative Picard number of~$f$ is $1$, every fiber of~$f$ is irreducible. By \cite[Prop.~21]{Ara09}, this implies that $f$ is a \emph{geometric quadric bundle,} i.e., there exists a rank-5 bundle $E$ on~$C$ such that $X$ embeds into~$\bP(E)$ as a divisor of relative degree~$2$. Now if we assume $X$ to be $2$-Fano, $-K_X$ is ample, so for $m \gg 0$, if we take two general divisors $H_1$, $H_2\in\lvert -mK_X \rvert$, the intersection
\[
S:= H_1\cap H_2
\]
is a smooth surface.

We will now calculate $\ch_2(X).S$. We denote by $\pi\colon \bP(E) \to C$ the natural projection and let $\xi := \sO_{\bP(E)}(1)$. Then
\[
-K_{\bP(E)} = 5\xi+\pi^*(\det E^* -  K_C)
\]
and
\[
\ch_2(\bP(E)) = \tfrac{5}{2}\xi^2+\pi^*c_1(E^*).\xi
\]
(cf.~\cite[Lem.~4.1]{dJS06}).

Since $f$ is a geometric quadric bundle, there exists a line bundle $L$ on~$C$ such that
\[
X \in \lvert 2\xi + \pi^* L \rvert.
\]
In particular,
\[
-K_X = \bigl(3\xi+\pi^*(\det E^* -  L - K_C)\bigr)|_X
\]
and
\[
\begin{split}
\ch_2(X) &= \bigl(\tfrac{5}{2}\xi^2+\pi^*c_1(E^*).\xi-\tfrac{1}{2}(2\xi+\pi^*L)^2\bigr)|_X\\
&= \bigl(\tfrac{1}{2}\xi^2+\pi^*(c_1(E^*)-2L).\xi\bigr)|_X.
\end{split}
\]
We thus obtain
\[
\begin{split}
\ch_2(X).S &= \ch_2(X).(-mK_X)^2\\
&= m^2\Bigl(
\bigl(\tfrac{1}{2}\xi^2+\pi^*(c_1(E^*)-2L).\xi\bigr).\bigl(3\xi+\pi^*(c_1(E^*) -  L - K_C)\bigr)^2
\Bigr)|_X\\
&= m^2
\bigl(
\tfrac{9}{2}\xi^4 +\pi^*(12c_1(E^*)-21L-3K_C).\xi^3
\bigr).(2\xi+\pi^*L)\\
&= m^2\bigl(9\xi^5 + \pi^*(24c_1(E^*)-\tfrac{75}{2}L-6K_C).\xi^4\bigr)\\
&= m^2(15\deg E^* -\tfrac{75}{2}\deg L - 6 \deg K_C).
\end{split}
\]
Since $X$ is given by a section
\[
s\in H^0(S^2E\otimes L)\subset \Hom(E^*,E\otimes L),
\]
the discriminant locus of~$f$ is a divisor
\[
\Delta_f\in \lvert 2\det E+5 L\rvert.
\]
We can thus rewrite the result of the above calculation as
\[
\ch_2(X).S=-\tfrac{15}{2}m^2(\deg \Delta_f+\tfrac{4}{5}\deg K_C).
\]
We now argue by contradiction: We assume that $\ch_2(X).S >0$. Then since $\Delta_f$ is effective, we must have $\deg K_C <0$, i.e., $C\cong \bP_1$, and so
\[
\deg\Delta_f\in\{0,1\}.
\]
By tensorizing $E$ with a suitable line bundle, we can assume
\[
0\le \deg E \le 4.
\]
In the case $\deg\Delta_f=0$, we thus obtain $\deg E = \deg L = 0$. We first show that the line bundle
\[
\xi \otimes \pi^*\sO_{\bP_1}(-1)
\]
cannot have a section: Suppose the contrary, then we get an effective Divisor
\[
D \in \lvert \xi - \pi^*\sO_{\bP_1}(1) \rvert
\]
on $X$. Intersecting $D$ with a general member of~$\lvert -mK_X \rvert$ gives a surface in~$X$ and we have
\[
\begin{split}
\ch_2(X).D.(-mK_X) &= m\xi^3.(\xi-\pi^*\sO_{\bP_1}(1)).(3\xi+\pi^*\sO_{\bP_1}(2))\\
&= m\bigl(3\xi^5 - \pi^*\sO_{\bP_1}(1).\xi^4\bigr) = -m,
\end{split}
\]
which contradicts $X$ being $2$-Fano. So $H^0\bigl((\xi\otimes\pi^*\sO_{\bP_1}(-1))|_X\bigr)=0$ and from the exact sequence
\[
0 \longrightarrow \xi^*\otimes\pi^*\sO_{\bP_1}(-1) \longrightarrow \xi\otimes\pi^*\sO_{\bP_1}(-1)  \longrightarrow (\xi\otimes\pi^*\sO_{\bP_1}(-1))|_X \longrightarrow 0
\]
we conclude that also $H^0(\xi\otimes\pi^*\sO_{\bP_1}(-1))=0$. But this means that
\[
H^0(E\otimes\sO_{\bP_1}(-1)) = 0.
\]
Since $E$ splits as a direct sum of line bundles, we conclude that $E\cong \sO_{\bP_1}^{\oplus 5}$ and thus $X\cong \bP_1 \times Q_3$ is not $2$-Fano.

In the case $\deg\Delta_f = 1$, we have $\deg E = 3$ and $\deg L = -1$, so we obtain
\[
-K_X = 3 \xi|_X,
\]
thus $X$ has index~$3$, so $X$ is not $2$-Fano by~\cite[Thm.~1.3]{AC12}.
\end{proof}

\section{The birational case}\label{sec:bir}
In this section we study the case of a birational extremal contraction
\[
f\colon X \to Y,
\]
where $X$ is a $2$-Fano fourfold. Unfortunately, we do not have a general classification result for this situation. We start by studying the blow-up of a smooth surface inside a smooth fourfold:
\begin{lem}\label{lem:blowupsurf}
Let $p \colon X \to Z$ be the blow-up of a smooth fourfold~$Z$ along a smooth surface $W \subset Z$. Then, for any smooth curve $C \subset W$, the smooth surface $S:= p^{-1}(C) \subset X$ satisfies
\[
\ch_2(X).S = - \tfrac{1}{2}\deg N_{W|Z}|_C.
\]
\end{lem}
\begin{proof}
Since $S$ is a $\bP_1$-bundle, we have $\ch_2(S)=0$ by~\cite[Prop.~4.3]{dJS06}.
The sequence
\[
0 \longrightarrow T_S \longrightarrow T_{X}|_S \longrightarrow N_{S|X} \longrightarrow 0.
\]
then implies that
\[ \ch_2(X).S = \ch_2(N_{S|X}). \]
Now if we denote by~$E$ the exceptional divisor of~$p$, we have an exact sequence
\[
0 \longrightarrow N_{S|E} \longrightarrow N_{S|X} \longrightarrow N_{E|X}|_S \longrightarrow 0.
\]
Since $N_{S|E} = (p|_S)^*N_{C|W}$, we have $\ch_2(N_{S|E}) = 0$, so it remains to calculate $\ch_2(N_{E|X}|_S)$. To this end, we observe that $N_{E|X} = \sO_E(-1)$, where $\sO_E(1)$ is the tautological line bundle on ~$E\cong \bP(N_{W|Z}^*)$. Thus we obtain
\[
\ch_2(N_{E|X}|_S) = \tfrac{1}{2} (\sO_E(-1))^2.((p|_E)^*\sO_W(C)) = \tfrac{1}{2} c_1(N_{W|Z}^*).c_1(\sO_W(C)).\qedhere
\]
\end{proof}

We can use this Lemma to obtain the following result:
\begin{prop}
Let $f\colon X \to Y$ be a birational divisorial extremal contraction which maps its exceptional divisor $D\subset X$ to a point in~$Y$. Then $X$ is not $2$-Fano.
\end{prop}
\begin{proof}
We assume that $X$ is $2$-Fano.

As in~\cite[Prop.~5]{Tsu06}, we can choose an extremal ray $\mathbb{R}_+[C_0]$ on $X$ such that $D.C_0 > 0$. Then by the proof of~\cite[Prop.~5]{Tsu06} (cf.~also \cite[Prop.~3.1]{Cas09}), the associated extremal contraction $g\colon X \to Z$ is either a conic bundle or the blow-up of a smooth surface inside a smooth Fano fourfold. We already showed in Proposition~\ref{prop:conicbund} that the conic bundle case cannot occur, so we conclude that $Z$ must be a Fano fourfold and $g$ is the blow-up of some smooth surface $W \subset Z$.

We denote by $E\subset X$ the exceptional divisor of~$g$. Then we have $E \cong \bP(N_{W|Z}^*)$ with the tautological line bundle $\sO_E(1) := \sO_{\bP(N_{W|Z}^*)}(1)$ and the natural projection map
\[
\pi\colon E \to W.
\]

In the case $\rho(X) \ge 3$ this situation has been studied in~\cite{Fuj12}. Fujita gives a very explicit description of the possible configurations that can occur, but for our purposes it is sufficient to notice that from his Thm.~1.1 it easily follows that $\det N_{W|Z}$ is always ample, which contradicts our Lemma~\ref{lem:blowupsurf}.

So for the rest of the proof we can assume that $\rho(X)=2$. We consider the scheme-theoretic intersection
\[ V := D \cap E.\]
\begin{claim}\label{cl:red}
$V$ is a reduced section of~$\pi$, i.e., $\pi|_V$ is an isomorphism.
\end{claim}
If we assume Claim~\ref{cl:red} to be true, we immediately obtain that $D$ and $E$ intersect transversally and that $g|_D\colon D \to g(D)=: D'$ is an isomorphism. Now since $f$ is the contraction of an extremal ray, we have
\begin{equation}\label{eq:relpic}
\operatorname{Im}(\Pic X \to \Pic D) \cong \bZ
\end{equation}
(cf.~\cite[Thm.~2.1]{And85}). This implies in particular that $V=E|_D$ is an ample Cartier divisor on~$D$ and thus that $W$ is an ample Cartier divisor on $D' = g(D)$.

Since $\rho(X) = 2$, it follows that $\rho(Z) = 1$, so that $D'$ is an ample divisor on~$Z$. The adjunction formula now gives
\begin{equation}\label{eq:dbladj}
\sO_W(K_W) = (\omega_{D'}\otimes \sO_{D'}(W))|_W = \sO_{D'}(K_Z+D'+W)|_W.
\end{equation}

But since $Z$ and $W$ are smooth, we also have
\begin{equation}\label{eq:adjform}
\sO_W(K_W) = \sO_W(K_Z) \otimes \det N_{W|Z}.
\end{equation}
Comparing \eqref{eq:dbladj} and \eqref{eq:adjform}, we obtain
\[ \det N_{W|Z} = \sO_{D'}(D'+W)|_W, \]
thus $\det N_{W|Z}$ is ample by our previous considerations. This contradicts Lemma~\ref{lem:blowupsurf}.

It remains to prove Claim~\ref{cl:red}. To do this, we need to obtain more information about~$W$. We first note that since $X$ is $2$-Fano, we have
\[ \det N_{W|Z}^*.C \ge 1\quad\text{for any curve $C\subset W$} \]
by Lemma~\ref{lem:blowupsurf}. Since $Z$ is Fano, this implies by~\eqref{eq:adjform} that $-K_W$ is ample. Furthermore, \eqref{eq:adjform} gives
\[ -K_W.C \ge 2 \quad\text{for any curve $C\subset W$.} \]
In particular, $W$ cannot contain any $(-1)$-curve, so we conclude that either $W \cong \bP_2$ or $W \cong \bP_1 \times \bP_1$.

The map $f$ induces a contraction on~$E$ contracting $V$ to a point. We now consider the reduction
\[ \tilde{V} := V_{\mathrm{red}} = (D\cap E)_{\mathrm{red}}. \]
The restriction $g|_{\tilde{V}}$ is then a finite morphism, so $\tilde{V}$ is a multisection of~$\pi$. Now let $\ell \subset W$ be a general line. Then $f|_{\pi^{-1}(\ell)}$ must be the contraction of the minimal section of the Hirzebruch surface $\pi^{-1}(\ell)$. This implies that $\tilde{V}$ is indeed a section of~$\pi$. Furthermore, since $N_{\tilde{V}|E}^*|_\ell$ is ample, it follows that $N_{\tilde{V}|E}^*$ is ample.

Thus it only remains to show that $V$ is reduced. As a divisor on~$E$, we can write
\[ V = \lambda \tilde{V} \qquad\text{for some $\lambda \ge 1$.}\]
Claim~\ref{cl:red} is then proved if we show that $\lambda = 1$. In order to show this, we first use Ando's result on the classification of the general fiber of the exceptional divisor of an extremal contraction \cite[Thm.~2.1]{And85}. By this result we have one of the following cases:
\begin{enumerate}
\item $D \cong \bP_3$,
\item $D \cong Q_3$, or
\item $D$ is a Del Pezzo variety such that if $A$ is an ample generator of~$\operatorname{Im}(\Pic X \to \Pic D) \cong \bZ$, then
\begin{equation}\label{eq:delpezzo}
\omega_D \cong (A^*)^{\otimes 2}.
\end{equation}  
\end{enumerate}
In the first case ($D\cong \bP_3$), Tsukioka proves in~\cite[Lem.~2]{Tsu06} that $\lambda = 1$, but his argument also applies to the second case ($D\cong Q_3$): If $\lambda > 1$, then $V = D\cap E$ is singular in every point of~$\tilde{V}$, so in particular $T_D|_{\tilde{V}} \cong T_E|_{\tilde{V}}$ (observe that $D$ is smooth). This implies that $N_{\tilde{V}|D} \cong N_{\tilde{V}|E}$. This is a contradiction since $N_{\tilde{V}|D}$ is ample because $\rho(D)=1$, but $N_{\tilde{V}|E}$ is negative as seen above.

It remains to study the Del Pezzo case. Since by adjunction,
\[ \omega_D \cong \sO_D(K_X+D), \]
we have
\begin{equation}
\sO_D(K_X) \cong \sO_D(D)
\end{equation}
by~\eqref{eq:delpezzo}. We thus get
\begin{equation}\label{eq:drestv}
\sO_{\tilde{V}}(D) \cong \sO_{\tilde{V}}(K_X) \cong \sO_{\tilde{V}}(g^*K_Z+E) \cong \sO_E(-1)|_{\tilde{V}}\otimes \sO_W(K_Z).
\end{equation}
On the other hand, we have
\[
\sO_{\tilde{V}}(D) \cong \sO_{E}(V)|_{\tilde{V}} \cong \bigl(\sO_{E}(\tilde{V})|_{\tilde{V}}\bigr)^{\otimes \lambda},
\]
i.e., $\sO_{\tilde{V}}(D)$ is divisible by~$\lambda$.

Now since $\tilde{V}$ is a section of~$\pi$, it is given by an element in $H^0(\sO_E(1)\otimes \pi^*L)$ for some $L \in \Pic W$, or, equivalently, by a section
\[ s \in H^0(N_{W|Z}^*\otimes L) \]
without zeroes. The section $s$ induces a short exact sequence of vector bundles
\begin{equation}\label{eq:normseq}
 0 \longrightarrow L^* \longrightarrow N_{W|Z}^* \longrightarrow \det N_{W|Z}^*\otimes L \longrightarrow 0,
\end{equation}
from which we conclude that
\begin{equation}\label{eq:taut}
\sO_E(1)|_{\tilde{V}} \otimes L \cong N_{\tilde{V}|E} \cong \det N_{W|Z}^*\otimes L^{\otimes 2}.
\end{equation}
This implies together with~\eqref{eq:drestv} that
\begin{equation}\label{eq:donv}
 \sO_{\tilde{V}}(D) \cong \sO_W(K_Z) \otimes \det N_{W|Z} \otimes L^* \cong \sO_W(K_W) \otimes L^*.
\end{equation}
Since $\sO_D(-D)$ is ample, we can conclude that $\sO_W(-K_W)\otimes L$ is an ample line bundle.

We now first consider the case $W\cong\bP_2$. Then $-K_W = \sO_{\bP_2}(3)$. Since $Z$ is Fano and Lemma~\ref{lem:blowupsurf} holds, only two cases can occur by~\eqref{eq:adjform}:
\begin{enumerate}
\item $\sO_W(-K_Z) \cong \sO_{\bP_2}(1)$ and $\det N_{W|Z}^* \cong \sO_{\bP_2}(2)$, or
\item $\sO_W(-K_Z) \cong \sO_{\bP_2}(2)$ and $\det N_{W|Z}^* \cong \sO_{\bP_2}(1)$.
\end{enumerate}
If we let
\[
L \cong \sO_{\bP_2}(a),
\]
then we have $a\ge -2$ because $L- K_W$ is ample as shown above. We already showed above that $N_{\tilde{V}|E}^* \cong \det N_{W|Z}\otimes (L^*)^{\otimes 2}$ is ample, which yields
$\deg N_{W|Z}-2a > 0$, i.e., $a < \tfrac{1}{2}\deg N_{W|Z}$. In the two cases stated above, this means the following:
\begin{enumerate}
\item In this case, we have $-2 \le a < -1$, so $a=-2$ and thus by~\eqref{eq:donv}
\[ \sO_{\tilde{V}}(D) \cong \sO_{\bP_2}(1). \]
Since $\sO_{\tilde{V}}(D)$ is divisible by~$\lambda$, we get $\lambda=1$.
\item Here we must have $-2 \le a < -\tfrac{1}{2}$, so $a\in\{-2, -1\}$. If $a=-2$, we get $\lambda=1$ as in the first case. If $a=-1$, we get
\[ \sO_{E}(-1)|_{\tilde{V}} \cong \sO_{\tilde{V}} \]
by~\eqref{eq:taut}. This is a contradiction because
\[ \sO_{E}(-1)|_{\tilde{V}} \cong \sO_E(E)|_{\tilde{V}} \cong \sO_D(E)|_{\tilde{V}} \]
is ample by~\eqref{eq:relpic}
\end{enumerate}

It remains to consider the case $W\cong \bP_1\times\bP_1$. Then $-K_W \cong \sO_{\bP_1\times\bP_1}(2,2)$ and thus by~\eqref{eq:adjform}, we must have
\[ \sO_W(-K_Z) \cong \det N_{W|Z}^* \cong \sO_{\bP_1\times\bP_1}(1,1) \]
(again using Lemma~\ref{lem:blowupsurf} and the fact that $Z$ is Fano). If we let
\[
L \cong \sO_{\bP_1\times\bP_1}(a,b)
\]
and use again the fact that $L-K_W$ is ample, we have $a$, $b\ge -1$. From the ampleness of~$N_{\tilde{V}|E}^*$ we obtain $-1-2a > 0$ and $-1-2b > 0$, so $a=b=-1$ and thus by~\eqref{eq:donv}
\[
\sO_{\tilde{V}}(D) \cong \sO_{\bP_1\times\bP_1}(1,1),
\]
from which we again conclude $\lambda=1$.
\end{proof}

\bibliographystyle{amsalpha}
\bibliography{literatur}

\providecommand{\bysame}{\leavevmode\hbox to3em{\hrulefill}\thinspace}
\providecommand{\MR}{\relax\ifhmode\unskip\space\fi MR }
\providecommand{\MRhref}[2]{%
  \href{http://www.ams.org/mathscinet-getitem?mr=#1}{#2}
}
\providecommand{\href}[2]{#2}
\begin{thebibliography}{ARM12}

\bibitem[AC12]{AC12}
Carolina Araujo and Ana-Maria Castravet, \emph{Classification of 2-fano
  manifolds with high index}, arXiv:math/1206.1357 (2012).

\bibitem[And85]{And85}
Tetsuya Ando, \emph{On extremal rays of the higher-dimensional varieties},
  Invent. Math. \textbf{81} (1985), no.~2, 347--357.

\bibitem[Ara09]{Ara09}
Carolina Araujo, \emph{Identifying quadric bundle structures on complex
  projective varieties}, Geom. Dedicata \textbf{139} (2009), 289--297.

\bibitem[ARM12]{ARM12}
Carolina Araujo and Jos{\'e}~J. Ram{\'o}n-Mar{\'\i}, \emph{Flat deformations of
  $\mathbb{P}^n$}, arXiV:math/1212.3593 (2012).

\bibitem[AW97]{AW97}
Marco Andreatta and Jaros{\l}aw~A. Wi{{\'s}}niewski, \emph{A view on
  contractions of higher-dimensional varieties}, Algebraic geometry---{S}anta
  {C}ruz 1995, Proc. Sympos. Pure Math., vol.~62, Amer. Math. Soc., Providence,
  RI, 1997, pp.~153--183.

\bibitem[Cas09]{Cas09}
C.~Casagrande, \emph{On {F}ano manifolds with a birational contraction sending
  a divisor to a curve}, Michigan Math. J. \textbf{58} (2009), no.~3, 783--805.

\bibitem[dJHS11]{dJHS11}
A.~J. de~Jong, Xuhua He, and Jason~Michael Starr, \emph{Families of rationally
  simply connected varieties over surfaces and torsors for semisimple groups},
  Publ. Math. Inst. Hautes {\'E}tudes Sci. (2011), no.~114, 1--85.

\bibitem[dJS06]{dJS06}
A.~J. de~Jong and Jason~Michael Starr, \emph{A note on {F}ano manifolds whose
  second {C}hern character is positive}, arXiV:math/0602644 (2006).

\bibitem[dJS07]{dJS07}
A.~J. de~Jong and Jason Starr, \emph{Higher {F}ano manifolds and rational
  surfaces}, Duke Math. J. \textbf{139} (2007), no.~1, 173--183.

\bibitem[Fuj90]{Fuj90}
Takao Fujita, \emph{On del {P}ezzo fibrations over curves}, Osaka J. Math.
  \textbf{27} (1990), no.~2, 229--245.

\bibitem[Fuj12]{Fuj12}
Kento Fujita, \emph{Fano manifolds having (n-1,0)-type extremal rays with large
  picard number}, arXiv:math/1212.4977 (2012).

\bibitem[GHS03]{GHS03}
Tom Graber, Joe Harris, and Jason Starr, \emph{Families of rationally connected
  varieties}, J. Amer. Math. Soc. \textbf{16} (2003), no.~1, 57--67
  (electronic).

\bibitem[Kac97]{Kac97}
Yasuyuki Kachi, \emph{Extremal contractions from {$4$}-dimensional manifolds to
  {$3$}-folds}, Ann. Scuola Norm. Sup. Pisa Cl. Sci. (4) \textbf{24} (1997),
  no.~1, 63--131.

\bibitem[Mor82]{Mor82}
Shigefumi Mori, \emph{Threefolds whose canonical bundles are not numerically
  effective}, Ann. of Math. (2) \textbf{116} (1982), no.~1, 133--176.

\bibitem[Tsu06]{Tsu06}
Toru Tsukioka, \emph{Classification of {F}ano manifolds containing a negative
  divisor isomorphic to projective space}, Geom. Dedicata \textbf{123} (2006),
  179--186.

\end{thebibliography}
\end{document}